\documentclass{birkjour}
\usepackage{graphicx}
\usepackage{amssymb}

\usepackage{amsfonts}

\providecommand{\U}[1]{\protect\rule{.1in}{.1in}}
\newtheorem{theorem}{Theorem}

\newtheorem{corollary}[theorem]{Corollary}

\newtheorem{lemma}[theorem]{Lemma}

\newtheorem{proposition}[theorem]{Proposition}

\begin{document}
\title[A non-convex variational problem]
 {A non-convex variational problem appearing in a large deformation
elasticity problem}

\author[Voisei]{M. D. Voisei}

\address{%
Towson University\\
Department of Mathematics\\
Towson, Maryland, U.S.A.}

\email{mvoisei@towson.edu}.

\author[Z\u{a}linescu]{C. Z\u{a}linescu}
\address{University {}``Al.I.Cuza'' Ia\c{s}i\br Faculty of
Mathematics\br and\br Institute of Mathematics Octav Mayer\br Ia\c{s}i,
Romania}

\email{zalinesc@uaic.ro}



\date{April 7, 2010}

\maketitle
\begin{abstract}
A result concerning global extrema in a nonsmooth nonconvex variational
problem that appears in applications (e.g. in a large deformation
elasticity problem) is investigated in comparison with a result of
D.Y. Gao and R.W. Ogden. The tools used are elementary and the results
derived improve upon and correct a recent similar result, more precisely,
Theorem 4 of the paper {}``Closed-form solutions, extremality and
nonsmoothness criteria in a large deformation elasticity problem''
by the fore-mentioned authors.
\end{abstract}

\section{Introduction}

Starting in 1998 a new optimization theory called the {}``triality
theory'' has been the object of intense studies (see e.g.
\cite{Gaoo-book} and the references within) while numerous
applications of this theory have been published in prestigious
journals throughout the literature. This triality theory promises
fruitful results for a large class of optimization problems and is
based on a gap function also called the Gao--Strang complementary
gap function (see \cite{Gao/Strang:89}).

In this paper we investigate one such application of the triality
theory, namely a variational problem appearing in an elasticity problem
studied in \cite{Gao/Ogden:08}, we prove that its main result \cite[Th.\ 4]{Gao/Ogden:08}
is false and we correct that result using elementary arguments (see
Proposition \ref{prop2} below). The analysis of \cite[Th.\ 4]{Gao/Ogden:08}
is accompanied by several comments on the context, mathematical writing
manner, and arguments of \cite{Gao/Ogden:08} and by several counterexamples.

The plan of the paper is as follows. In section 2 we study a series
of polynomial and rational functions together with their critical
points, relative extrema, and behavior. Section 3 deals with the nonconvex
variational problem in focus. The main object of Section 4 is to perform
a comparison of the results in Section 3 with \cite[Th.\ 4]{Gao/Ogden:08}.
Section 5 presents our main conclusions.

\section{An elementary argument}

We begin with an elementary study of some simple functions. Throughout
this paper $\alpha,\mu,\nu>0$ and $\tau\in\mathbb{R}$; we consider
the polynomial \begin{equation}
p_{\tau}(y):=\tfrac{1}{2}\mu y^{2}+\tfrac{1}{2}\nu\left(\tfrac{1}{2}y^{2}-\alpha y\right)^{2}-(\tau+\alpha\mu)y\label{def-p}\end{equation}
 and the functions \begin{equation}
h_{\tau}:\mathbb{R}\setminus\{-\mu\}\rightarrow\mathbb{R},\quad h_{\tau}(\varsigma):=-\frac{1}{2}\left(\frac{\tau^{2}}{\varsigma+\mu}+2\alpha\tau+\alpha^{2}(\varsigma+\mu)+\nu^{-1}\varsigma^{2}\right)\label{def-h}\end{equation}
 for $\tau\in\mathbb{R}\setminus\{0\}$, and \begin{equation}
h_{0}:\mathbb{R}\rightarrow\mathbb{R},\quad h_{0}(\varsigma):=-\tfrac{1}{2}\left(\alpha^{2}(\varsigma+\mu)+\nu^{-1}\varsigma^{2}\right).\label{def-h0}\end{equation}

For $\tau\in\mathbb{R}$ we have \begin{equation}
h_{\tau}^{\prime}(\varsigma)=-\frac{1}{2}\left(-\frac{\tau^{2}}{(\varsigma+\mu)^{2}}+\alpha^{2}+2\nu^{-1}\varsigma\right)=-\frac{1}{2}\frac{f(\varsigma)-\tau^{2}}{(\varsigma+\mu)^{2}}\quad\forall\varsigma\in\mathbb{R}\setminus\{-\mu\},\label{der-h}\end{equation}
 where $f$ is the polynomial\[
f(\varsigma):=\left(2\nu^{-1}\varsigma+\alpha^{2}\right)(\mu+\varsigma)^{2}.\]
 For $\tau\in\mathbb{R}$, let us consider the function $\Xi_{\tau}:\mathbb{R}^{2}\rightarrow\mathbb{R}$
defined by \[
\Xi_{\tau}(u,\varsigma):=\tfrac{1}{2}u^{2}(\varsigma+\mu)-\alpha u\varsigma-\tfrac{1}{2}\nu^{-1}\varsigma^{2}-(\tau+\alpha\mu)u.\]

From the expression of $\Xi_{\tau}$ we see that, for every fixed
$\tau\in\mathbb{R}$, $\Xi_{\tau}(u,\cdot)$ is concave for every
$u\in\mathbb{R}$ and $\Xi_{\tau}(\cdot,\varsigma)$ is convex (concave)
for $\varsigma\geq-\mu$ $(\varsigma\leq-\mu).$

We have\begin{gather}
p_{\tau}^{\prime}(y)=\mu y+\nu\left(\tfrac{1}{2}y^{2}-\alpha y\right)(y-\alpha)-\tau-\alpha\mu,\nonumber \\
p_{\tau}^{\prime\prime}(y)=\mu+\nu\left(\tfrac{3}{2}y^{2}-3\alpha y+\alpha^{2}\right),\nonumber \\
\frac{\partial\Xi_{\tau}}{\partial u}(u,\varsigma)=u(\varsigma+\mu)-\alpha\varsigma-\tau-\alpha\mu,\nonumber \\
\frac{\partial\Xi_{\tau}}{\partial\varsigma}(u,\varsigma)=\tfrac{1}{2}u^{2}-\alpha u-\nu^{-1}\varsigma.\label{dXidvars}\end{gather}

Then $(u,\varsigma)$ is a critical point of $\Xi_{\tau}$ iff\begin{equation}
u(\varsigma+\mu)-\alpha\varsigma-\tau-\alpha\mu=0,\quad\tfrac{1}{2}u^{2}-\alpha u-\nu^{-1}\varsigma=0.\label{cp-Xi}\end{equation}

\begin{theorem} \label{Thm1}Let $p_{\tau}$, $h_{\tau}$, $\Xi_{\tau}$
and $f$ be defined as above.

\emph{(i)} If $(u,\varsigma)$ is a critical point of $\Xi_{\tau}$
then $\varsigma=\nu\left(\tfrac{1}{2}u^{2}-\alpha u\right)$, $p_{\tau}^{\prime}(u)=0$,
i.e. $u$ is a critical point of $p_{\tau}$, \begin{equation}
p_{\tau}^{\prime\prime}(u)=\mu+\nu(3\nu^{-1}\varsigma+\alpha^{2})=3(\varsigma-\rho),\label{d2pt-pc}\end{equation}
 where \begin{equation}
\rho:=-\tfrac{1}{3}(\mu+\nu\alpha^{2}),\label{rho}\end{equation}
 $\varsigma$ is a solution of the equation \begin{equation}
f(\varsigma)=\tau^{2}\label{ec-f}\end{equation}
 (in particular, when $\tau\ne0$, $\varsigma$ is a critical point
of $h_{\tau}$), and

\begin{equation}
p_{\tau}(u)=\Xi_{\tau}(u,\varsigma)=h_{\tau}(\varsigma).\label{r-dual}\end{equation}

\emph{(ii)} For every $u\in\mathbb{R}$, set \begin{equation}
\varsigma_{u}:=\nu\left(\tfrac{1}{2}u^{2}-\alpha u\right).\label{zetav}\end{equation}

Then $\varsigma_{u}$ is a global maximum point for $\Xi_{\tau}(u,\cdot)$
and\begin{equation}
p_{\tau}(u)=\Xi_{\tau}(u,\varsigma_{u})=\sup_{\varsigma\in\mathbb{R}}\Xi_{\tau}(u,\varsigma)\quad\forall u\in\mathbb{R}.\label{u-gen}\end{equation}

\emph{(iii)} If $p_{\tau}^{\prime}(u)=0$ then $(u,\varsigma_{u})$
is a critical point of $\Xi_{\tau}$, $\varsigma_{u}$ is a solution
of (\ref{ec-f}), and (\ref{r-dual}) holds for $\varsigma=\varsigma_{u}.$

\emph{(iv)} For $\varsigma\neq-\mu$ set \begin{equation}
u_{\varsigma}:=\alpha+\tau/\left(\varsigma+\mu\right).\label{vzeta}\end{equation}
 Then\begin{equation}
h_{\tau}(\varsigma)=\Xi_{\tau}(u_{\varsigma},\varsigma)=\left\{ \begin{array}{ccc}
\inf_{u\in\mathbb{R}}\Xi_{\tau}(u,\varsigma) & \text{if} & \varsigma>-\mu,\\
\sup_{u\in\mathbb{R}}\Xi_{\tau}(u,\varsigma) & \text{if} & \varsigma<-\mu,\end{array}\right.\label{pds}\end{equation}
 that is, if $\varsigma+\mu>0$ $(\varsigma+\mu<0)$ then $u_{\varsigma}$
is a global minimum (maximum) point of $\Xi_{\tau}(\cdot,\varsigma)$.

\emph{(v)} If $\varsigma\neq-\mu$ is a solution of (\ref{ec-f})
(or equivalently, $\varsigma$ is a critical point of $h_{\tau}$)
then $(u_{\varsigma},\varsigma)$ is a critical point of $\Xi_{\tau},$
$u_{\varsigma}$ is a critical point of $p_{\tau}$ and (\ref{r-dual})
holds for $u=u_{\varsigma}$.

\emph{(vi)} If $(\overline{u},\overline{\varsigma})\in\mathbb{R}^{2}$
is a critical point of $\Xi_{\tau}$ with $\overline{\varsigma}>-\mu$
then $\overline{\varsigma}=\varsigma_{\overline{u}}$, $\overline{u}=u_{\overline{\varsigma}}$,
\begin{equation}
\sup_{\varsigma\in\mathbb{R}}\inf_{u\in\mathbb{R}}\Xi_{\tau}(u,\varsigma)=\inf_{u\in\mathbb{R}}\Xi_{\tau}(u,\overline{\varsigma})=\Xi_{\tau}(\overline{u},\overline{\varsigma})=p_{\tau}(\overline{u})=\inf_{u\in\mathbb{R}}p_{\tau}(u)\label{minmax}\end{equation}
 and\begin{equation}
\sup_{\varsigma>-\mu}h_{\tau}(\varsigma)=h_{\tau}(\overline{\varsigma})=\Xi_{\tau}(\overline{u},\overline{\varsigma})=p_{\tau}(\overline{u})=\inf_{u\in\mathbb{R}}p_{\tau}(u).\label{supinf}\end{equation}
 In particular, $\overline{u}$ is a global minimum of $p_{\tau}$
on $\mathbb{R}$.

\end{theorem}

\begin{proof} (i) Let $(u,\varsigma)$ be a critical point of $\Xi_{\tau}$.
Relation $\varsigma=\nu\left(\tfrac{1}{2}u^{2}-\alpha u\right)$ follows
directly from the second part in (\ref{cp-Xi}). This yields $p_{\tau}^{\prime}(u)=\mu u+\varsigma(y-\alpha)-\tau-\alpha\mu=0$
by the first part in (\ref{cp-Xi}) and $p_{\tau}^{\prime\prime}(u)=\mu+\nu(3\nu^{-1}\varsigma+\alpha^{2})=3(\varsigma-\rho)$
where $\rho=-\tfrac{1}{3}(\mu+\nu\alpha^{2})$. Again, $\varsigma=\nu\left(\tfrac{1}{2}u^{2}-\alpha u\right)$
becomes $2\nu^{-1}\varsigma+\alpha^{2}=(u-\alpha)^{2}$ while the
first equality in (\ref{cp-Xi}) is equivalent to $(u-\alpha)(\varsigma+\mu)=\tau$.
These easily provide $f(\varsigma)=(2\nu^{-1}\varsigma+\alpha^{2})(\varsigma+\mu)^{2}=\tau^{2}$.
When $\tau\ne0$ we clearly have that $\varsigma\not=-\mu$. Hence,
according to (\ref{der-h}) and (\ref{ec-f}), $\varsigma$ is a critical
point of $h_{\tau}$.

A direct computation based on the relation $\varsigma=\nu\left(\tfrac{1}{2}u^{2}-\alpha u\right)$
provides the equality $p_{\tau}(u)=\Xi_{\tau}(u,\varsigma)$. For
$\Xi_{\tau}(u,\varsigma)=h_{\tau}(\varsigma)$ first notice again
that $\tau=0$ whenever $\varsigma=-\mu$ and $\Xi_{0}(u,-\mu)=-\tfrac{1}{2}\nu^{-1}\mu^{2}=h_{0}(-\mu)$.
It remains to prove the last equality in (\ref{r-dual}) for $\varsigma\ne-\mu$
(which happens, in particular, when $\tau\ne0$). In this case $u=\alpha+\frac{\tau}{\varsigma+\mu}$,
$\frac{\tau^{2}}{\varsigma+\mu}=\tau(u-\alpha)$ and the equality
$\Xi_{\tau}(u,\varsigma)=h_{\tau}(\varsigma)$ reduces to $(u-\alpha)^{2}(\varsigma+\mu)=\tau(u-\alpha)$
which is clearly true.

(ii) It is easily seen from (\ref{dXidvars}) and (\ref{zetav}) that,
for every $u\in\mathbb{R}$, $\varsigma_{u}$ is a critical point
for the concave function $\Xi_{\tau}(u,\cdot)$; hence $\varsigma_{u}$
is a global maximum point of $\Xi_{\tau}(u,\cdot)$. This fact is
reflected by the second equality in (\ref{u-gen}). The first equality
in (\ref{u-gen}) follows directly from (\ref{zetav}).

(iii) Since $p_{\tau}'(u)=\mu u+\varsigma_{u}(u-\alpha)-\tau-\alpha\mu=0$
one sees that $(u,\varsigma_{u})$ satisfies (\ref{cp-Xi}). The second
part follows from (i).

(iv) For (\ref{pds}) one uses again the fact that a critical point
for a convex (concave) functional is a global minimum (maximum) point
for that functional; apply this for $u_{\varsigma}$ which is a critical
point of $\Xi(\cdot,\varsigma)$.

(v) While the first part in (\ref{cp-Xi}) is straightforward due
to (\ref{vzeta}) the second part part in (\ref{cp-Xi}) follows from
$f(\varsigma)=\tau^{2}$ coupled again with $\tau/(\varsigma+\mu)=u_{\varsigma}-\alpha$.
Therefore $(u_{\varsigma},\varsigma)$ is a critical point of $\Xi_{\tau}$,
and from (i) we have that $u_{\varsigma}$ is a critical point of
$p_{\tau}$ and (\ref{r-dual}) holds for $\varsigma$ and $u=u_{\varsigma}$.

(vi) Let $(\overline{u},\overline{\varsigma})\in\mathbb{R}^{2}$ be
a critical point of $\Xi_{\tau}$ with $\overline{\varsigma}>-\mu$.
Relations $\overline{\varsigma}=\varsigma_{\overline{u}}$ and $\overline{u}=u_{\overline{\varsigma}}$
are consequences of (i). Since $\overline{u}$ is a critical point
for the convex function $\Xi_{\tau}(\cdot,\overline{\varsigma})$
we know that $\overline{u}$ is a global minimum point of $\Xi_{\tau}(\cdot,\overline{\varsigma})$.
Similarly, $\overline{\varsigma}$ is a critical point for the concave
function of $\Xi_{\tau}(\overline{u},\cdot)$ thus $\overline{\varsigma}$
is a global maximum point of $\Xi_{\tau}(\overline{u},\cdot)$. These
facts translate as

\[
\Xi_{\tau}(u,\overline{\varsigma})\geq\Xi_{\tau}(\overline{u},\overline{\varsigma})\geq\Xi_{\tau}(\overline{u},\varsigma)\quad\forall u\in\mathbb{R},forall\varsigma\in\mathbb{R}.\]
 It follows that \[
\sup_{\varsigma\in\mathbb{R}}\inf_{u\in\mathbb{R}}\Xi_{\tau}(u,\varsigma)\geq\inf_{u\in\mathbb{R}}\Xi_{\tau}(u,\overline{\varsigma})=\Xi_{\tau}(\overline{u},\overline{\varsigma})=\sup_{\varsigma\in\mathbb{R}}\Xi_{\tau}(\overline{u},\varsigma)\geq\inf_{u\in\mathbb{R}}\sup_{\varsigma\in\mathbb{R}}\Xi_{\tau}(u,\varsigma).\]
 Since $\sup_{\varsigma\in\mathbb{R}}\inf_{u\in\mathbb{R}}\Xi_{\tau}(u,\varsigma)\leq\inf_{u\in\mathbb{R}}\sup_{\varsigma\in\mathbb{R}}\Xi_{\tau}(u,\varsigma)$
(for every function $\Xi_{\tau}$), we obtain that\[
\sup_{\varsigma\in\mathbb{R}}\inf_{u\in\mathbb{R}}\Xi_{\tau}(u,\varsigma)=\inf_{u\in\mathbb{R}}\Xi_{\tau}(u,\overline{\varsigma})=\Xi_{\tau}(\overline{u},\overline{\varsigma})=p_{\tau}(\overline{u})=\inf_{u\in\mathbb{R}}p_{\tau}(u);\]
 the last equality being due to (\ref{u-gen}). Taking into account
(\ref{pds}), (\ref{minmax}), and (v) we get \[
\sup_{\varsigma>-\mu}h_{\tau}(\varsigma)=h_{\tau}(\overline{\varsigma})=\Xi_{\tau}(\overline{u},\overline{\varsigma})=p_{\tau}(\overline{u})=\inf_{u\in\mathbb{R}}p_{\tau}(u),\]
 that is, (\ref{supinf}) holds. \end{proof}

\medskip{}

The following result analyzes the solutions of (\ref{ec-f}).

\begin{proposition} \label{prop1}Assume that $2\mu<\nu\alpha^{2}$
and set \begin{equation}
\eta:=(\nu\alpha^{2}-2\mu)^{3}/27\nu.\label{eta}\end{equation}

\emph{(a)} If $\tau^{2}>\eta$ then equation (\ref{ec-f}) has a unique
real solution $\overline{\varsigma};$ moreover, $\overline{\varsigma}>-\mu$
and for $\overline{u}:=u_{\overline{\varsigma}}$ one has (\ref{r-v1})
below with $\overline{u}$ and $\overline{\varsigma}$ instead of
$\overline{u}_{1}$ and $\overline{\varsigma}_{1}$, respectively.

\emph{(b)} If $\tau^{2}=\eta$ then equation (\ref{ec-f}) has the
real solutions $\overline{\varsigma}_{1}>-\mu$ and $\overline{\varsigma}_{2}=\overline{\varsigma}_{3}=\rho$,
where $\rho$ is defined in (\ref{rho}). Moreover, for $\overline{u}_{1}:=u_{\overline{\varsigma}_{1}}$
one has \begin{equation}
p_{\tau}(u)>p_{\tau}(\overline{u}_{1})=h_{\tau}(\overline{\varsigma}_{1})>h_{\tau}(\varsigma)\quad\forall u\in\mathbb{R}\setminus\{\overline{u}_{1}\},forall\varsigma\in(-\mu,\infty)\setminus\{\overline{\varsigma}_{1}\}.\label{r-v1}\end{equation}

\emph{(c)} If $0<\tau^{2}<\eta$ then equation (\ref{ec-f}) has the
real solutions $\overline{\varsigma}_{1},\overline{\varsigma}_{2},\overline{\varsigma}_{3}$
with \[
\overline{\varsigma}_{1}>-\mu>\overline{\varsigma}_{2}>\rho>\overline{\varsigma}_{3}>-\tfrac{1}{2}\nu\alpha^{2}.\]
 Moreover, for $\overline{u}_{i}:=u_{\overline{\varsigma}_{i}}$ $(i\in\{1,2,3\}),$
one has (\ref{r-v1}) and for $i\in\{2,3\}$ there exists a neighborhood
$U_{i}$ of $\overline{u}_{i}$ such that and \begin{equation}
\min_{u\in U_{2}}p_{\tau}(u)=p_{\tau}(\overline{u}_{2})=h_{\tau}(\overline{\varsigma}_{2})=\min_{\varsigma\in(\overline{\varsigma}_{3},-\mu)}h_{\tau}(\varsigma)\label{z2v2}\end{equation}
 and\begin{equation}
\max_{u\in U_{3}}p_{\tau}(u)=p_{\tau}(\overline{u}_{3})=h_{\tau}(\overline{\varsigma}_{3})=\max_{\varsigma\in(-\infty,\overline{\varsigma}_{2})}h_{\tau}(\varsigma).\label{z3v3}\end{equation}

\emph{(d)} If $\tau=0$ then equation (\ref{ec-f}) has the real solutions
$\overline{\varsigma}_{1}=\overline{\varsigma}_{2}=-\mu$ and $\overline{\varsigma}_{3}=-\tfrac{1}{2}\nu\alpha^{2}$.
Setting $\overline{u}_{1}=\alpha+\sqrt{\alpha^{2}-2\nu^{-1}\mu},$
$\overline{u}_{2}=\alpha-\sqrt{\alpha^{2}-2\nu^{-1}\mu}$ and $\overline{u}_{3}:=u_{\overline{\varsigma}_{3}}=\alpha$,
we have that \begin{equation}
\min_{u\in\mathbb{R}}p_{0}(u)=p_{0}(\overline{u}_{1})=p_{0}(\overline{u}_{2})=\sup_{\varsigma>-\mu}h_{0}(\varsigma)=\inf_{\mu-\nu\alpha^{2}<\varsigma<-\mu}h_{0}(\varsigma)=-\tfrac{1}{2}\nu^{-1}\mu^{2},\label{tau0}\end{equation}
 and (\ref{z3v3}) holds with $U_{3}:=(\overline{u}_{2},\overline{u}_{1})$.
\end{proposition}

\begin{proof}We have and $f'(\varsigma)=6\nu^{-1}(\varsigma+\mu)(\varsigma-\rho)$.
Also, $-\tfrac{1}{2}\nu\alpha^{2}<\rho<-\mu$ since $2\mu<\nu\alpha^{2}$
and $f(\rho)=\eta$. The behavior of $f$ is shown in Table \ref{tab1}.

\begin{table}[h]
 \center \begin{tabular}{c|ccccccccc}
$\varsigma$  & $-\infty$  &  & $-\tfrac{1}{2}\nu\alpha^{2}$  &  & $\rho$  &  & $-\mu$  &  & $+\infty$\tabularnewline
\hline
$f^{\prime}(\varsigma)$  &  & $+$  & $+$  & $+$  & $0$  & $-$  & $0$  & $+$  & \tabularnewline
\hline
$f(\varsigma)$  & $-\infty$  & $\nearrow$  & $0$  & $\nearrow$  & $\eta$  & $\searrow$  & $0$  & $\nearrow$  & $+\infty$\tabularnewline
\end{tabular}

\caption{The behavior of $f$.}

\label{tab1}
\end{table}

Consider the polynomial equation $f(\varsigma)=\tau^{2};$ it has
a unique real solution $\varsigma_{1}>-\mu$ for $\tau^{2}>\eta$,
three real solutions $\varsigma_{1}>-\mu$ and $\varsigma_{2}=\varsigma_{3}=\rho$
for $\tau^{2}=\eta$, three real solutions $\varsigma_{1}>-\mu>\varsigma_{2}>\rho>\varsigma_{3}>-\tfrac{1}{2}\nu\alpha^{2}$
for $0<\tau^{2}<\eta$, and three real solutions $\varsigma_{1}=\varsigma_{2}=-\mu$
and $\varsigma_{3}=-\tfrac{1}{2}\nu\alpha^{2}$ for $\tau^{2}=0$
(see Figure \ref{tab-extra} below).

\begin{figure}[h]
\center\includegraphics[scale=0.7]{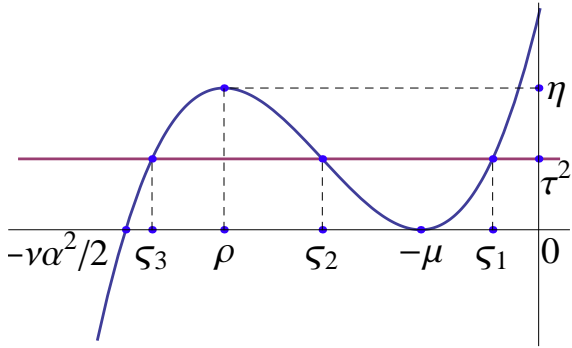}

\caption{The graph of $f$ and solutions of $f(\varsigma)=\tau^{2}$.}

\label{tab-extra}
\end{figure}

Based on the derivative of $h_{\tau}$ given in (\ref{der-h}), the
behavior of $h_{\tau}$ is presented in Table \ref{tab2} for $\tau^{2}>\eta$
and in Table \ref{tab3} for $0<\tau^{2}\leq\eta$.

\begin{table}[h]
 \center\begin{tabular}{c|ccccccc}
$\varsigma$  & $-\infty$  &  & $-\mu$  &  & $\varsigma_{1}$  &  & $+\infty$\tabularnewline
\hline
$h_{\tau}^{\prime}(\varsigma)$  &  & $+$  & $|$  & $+$  & $0$  & $-$  & $0$\tabularnewline
\hline
$h_{\tau}(\varsigma)$  & $-\infty$  & $\nearrow$  & $^{+\infty}|_{-\infty}$  & $\nearrow$  & $h_{\tau}(\varsigma_{1})$  & $\searrow$  & $-\infty$\tabularnewline
\end{tabular}

\caption{The behavior of $h$ for $\tau^{2}>\eta$.}

\label{tab2}
\end{table}

\begin{table}[h]
 \center\begin{tabular}{cccccccccccc}
$\varsigma$  & $-\infty$  &  & $\varsigma_{3}$  &  & $\varsigma_{2}$  &  & $-\mu$  &  & $\varsigma_{1}$  &  & $+\infty$\tabularnewline
\hline
$h_{\tau}^{\prime}(\varsigma)$  & $+$  & $+$  & $0$  & $-$  & $0$  & $+$  & $|$  & $+$  & $0$  & $-$  & \tabularnewline
\hline
$h_{\tau}(\varsigma)$  & $-\infty$  & $\nearrow$  & $h_{\tau}(\varsigma_{3})$  & $\searrow$  & $h_{\tau}(\varsigma_{2})$  & $\nearrow$  & $^{+\infty}|_{-\infty}$  & $\nearrow$  & $h_{\tau}(\varsigma_{1})$  & $\searrow$  & $-\infty$\tabularnewline
\end{tabular}

\caption{The behavior of $h$ for $0<\tau^{2}\leq\eta$.}

\label{tab3}
\end{table}

Assume that $\tau\neq0$. From the discussion above we have that the
equation $f(\varsigma)=\tau^{2}$ has a unique solution $\overline{\varsigma}_{1}:=\varsigma_{1}$
on the interval $(-\mu,\infty)$. Since $\overline{\varsigma}_{1}>-\mu$,
from Theorem \ref{Thm1} (v), (vi) we know that $(\overline{u}_{1},\overline{\varsigma}_{1})$
with $\overline{u}_{1}:=u_{\overline{\varsigma}_{1}}$ is a critical
point of $\Xi_{\tau}$ and \[
\inf_{u\in\mathbb{R}}p_{\tau}(u)=p_{\tau}(\overline{u}_{1})=h_{\tau}(\overline{\varsigma}_{1})=\sup_{\varsigma>-\mu}h_{\tau}(\varsigma).\]
 The fact that $h_{\tau}(\overline{\varsigma}_{1})>h_{\tau}(\varsigma)$
for $\varsigma\in(-\mu,\infty)\setminus\{\overline{\varsigma}_{1}\}$
is clear from Tables \ref{tab2}, \ref{tab3}. In order to complete
(\ref{r-v1}) we have to prove that $\overline{u}_{1}$ is the only
(strict) global minimum point for $p_{\tau}$, i.e., $p_{\tau}(u)>p_{\tau}(\overline{u}_{1})$
for $u\in\mathbb{R}\setminus\{\overline{u}_{1}\}.$

Assume that there exists $\overline{u}_{0}\in\mathbb{R}\setminus\{\overline{u}_{1}\}$
such that $p_{\tau}(\overline{u}_{0})=p_{\tau}(\overline{u}_{1})$.
Hence $\overline{u}_{0}$ is a global minimum point of $p_{\tau}$.
The polynomial $q(y)=p_{\tau}(y)-p_{\tau}(\overline{u}_{0})$ has
degree 4, is non-negative, and admits the distinct roots $\overline{u}_{0},\overline{u}_{1}$.
Hence $\overline{u}_{0},\overline{u}_{1}$ are double roots and so
$q(y)=\tfrac{1}{8}\nu(y-\overline{u}_{0})^{2}(y-\overline{u}_{1})^{2}$
after one takes into account the leading coefficient of $p_{\tau}$.
It follows that $p_{\tau}(y)=p_{\tau}(\overline{u}_{0})+\tfrac{1}{8}\nu(y-\overline{u}_{0})^{2}(y-\overline{u}_{1})^{2}$.
In expanded form, \[
p_{\tau}(y)=\tfrac{1}{8}\nu y^{4}-\tfrac{1}{2}\nu\alpha y^{3}+(\tfrac{1}{2}\nu\alpha^{2}+\tfrac{1}{2}\mu)y^{2}-(\tau+\alpha\mu)y+p_{\tau}(0).\]
 After we identify the coefficients of $y^{3},y^{2}$ and $y$ we
find $\overline{u}_{0}+\overline{u}_{1}=2\alpha$, $\overline{u}_{0}\overline{u}_{1}=2\mu\nu^{-1}$,
and $\tau=0$. Since by our running hypothesis $\tau\neq0$, we obtain
$p_{\tau}(u)>p_{\tau}(\overline{u}_{0})$ for all $\mathbb{R}\setminus\{\overline{u}_{1}\}.$

(a), (b) and the first part of (c) follow from the discussion above
on $f$ and $\overline{\varsigma}_{1}.$

(c) Assume that $0<\tau^{2}<\eta$. The fact that $-\mu>\overline{\varsigma}_{2}>\rho>\overline{\varsigma}_{3}>-\tfrac{1}{2}\nu\alpha^{2}$
follows from Table \ref{tab-extra}. Also from the behavior of $h_{\tau}$
shown in Table \ref{tab3} we have that \[
h_{\tau}(\overline{\varsigma}_{2})=\min_{\varsigma\in(\overline{\varsigma}_{3},-\mu)}h_{\tau}(\varsigma),\quad h_{\tau}(\overline{\varsigma}_{3})=\max_{\varsigma\in(-\infty,\overline{\varsigma}_{2})}h_{\tau}(\varsigma).\]
 According to Theorem \ref{Thm1} $(\overline{u}_{i},\overline{\varsigma}_{i})$
is a critical point of $\Xi_{\tau}$, $p_{\tau}^{\prime}(\overline{u}_{i})=0$,
$h_{\tau}(\overline{\varsigma}_{i})=p_{\tau}(\overline{u}_{i})$ and
$p_{\tau}^{\prime\prime}(\overline{u}_{i})=3\left(\overline{\varsigma}_{i}-\rho\right)$
for $i\in\{2,3\}$. Because $\overline{\varsigma}_{2}>\rho>\overline{\varsigma}_{3}$,
we have that $p_{\tau}^{\prime\prime}(\overline{u}_{2})>0$ and $p_{\tau}^{\prime\prime}(\overline{u}_{3})<0$,
proving that $\overline{u}_{2}$ is a strict local minimum point of
$p_{\tau}$ and $\overline{u}_{3}$ is a strict local maximum point
of $p_{\tau}$.

(d) For $\tau=0$ we have \[
p_{0}(y)=\tfrac{1}{8}\nu\left(y^{2}-2\alpha y+2\nu^{-1}\mu\right)^{2}-\tfrac{1}{2}\nu^{-1}\mu^{2}=\tfrac{1}{8}\nu\left(y-\overline{u}_{1}\right)^{2}\left(y-\overline{u}_{2}\right)^{2}-\tfrac{1}{2}\nu^{-1}\mu^{2}.\]
 Taking into account that $h_{0}(\varsigma)=-\frac{1}{2}\left[\alpha^{2}(\varsigma+\mu)+\nu^{-1}\varsigma^{2}\right]$
for $\varsigma\in\mathbb{R}$, relation (\ref{tau0}) and the rest
of the conclusion are straightforward. \end{proof}

\begin{corollary} \label{cor1}Assume that $2\mu<\nu\alpha^{2}$.
With the notation of the preceding proposition, if $0<\tau^{2}<\eta$
then $p_{\tau}(\overline{u}_{3})>p_{\tau}(\overline{u}_{2})>p_{\tau}(\overline{u}_{1}).$
\end{corollary}

\begin{proof} Since $ $$\overline{u}_{2}\neq\overline{u}_{1}$,
from (\ref{r-v1}) we get $p_{\tau}(\overline{u}_{2})>p_{\tau}(\overline{u}_{1})$.
Also, from (\ref{z2v2}), (\ref{z3v3}), and Table \ref{tab3} $p_{\tau}(\overline{u}_{3})=h_{\tau}(\overline{\varsigma}_{3})>h_{\tau}(\overline{\varsigma}_{2})=p_{\tau}(\overline{u}_{2})$.
\end{proof}

\section{Application to a nonconvex variational problem\label{sec2}}

Let us consider $\tau_{\theta},\alpha,\mu>0$, $0<a<b,$ $\sigma:[a,b]\rightarrow\mathbb{R}$
be defined by $\sigma(r):=b^{2}\tau_{\theta}r^{-2}$ and $\widehat{\mathcal{P}}:\mathcal{L}^{4}(a,b)\rightarrow\mathbb{R}$
be defined by \begin{equation}
\widehat{\mathcal{P}}(v):=2\pi\int_{a}^{b}\left[\tfrac{1}{2}\mu r^{2}v^{2}+\tfrac{1}{2}\nu\left(\tfrac{1}{2}r^{2}v^{2}-\alpha rv\right)^{2}-\sigma v\right]rdr=2\pi\int_{a}^{b}rp_{-\beta}(rv)dr,\label{ph}\end{equation}
 where $v\in\mathcal{L}^{4}(a,b)$, $\beta(r):=\alpha\mu-\sigma(r)$,
$r\in[a,b]$, and $p_{-\beta}$ is defined via (\ref{def-p}). In
the previous formula and the subsequent ones we use $v,\sigma$, and
$\beta$ instead of $v(r),\sigma(r)$, and $\beta(r)$ for the simplicity
of notation.

Throughout this note we use the notation $\mathcal{L}^{p}:=\mathcal{L}^{p}(a,b)$,
$p\ge1$, and the convention $0/0:=0$, which agrees with the convention
$0\cdot(\pm\infty):=0$ used in measure theory. With this convention
in mind, consider

\begin{equation}
A_{1}:=\left\{ \zeta\in\mathcal{L}^{2}\,\bigg|\,\frac{\beta^{2}}{\zeta+\mu}\in\mathcal{L}^{1}\right\} ;\label{r-a1}\end{equation}
 and the function \begin{equation}
\mathcal{P}^{d}:A_{1}\rightarrow\mathbb{R},\quad\mathcal{P}^{d}(\zeta):=-\pi\int_{a}^{b}\left(\frac{(\sigma+\alpha\zeta)^{2}}{\zeta+\mu}+\nu^{-1}\zeta^{2}\right)rdr.\label{pd-first}\end{equation}
 Taking into account (\ref{pdsc}) below, $A_{1}$ is the biggest
subset of $\mathcal{L}^{2}$ for which $\mathcal{P}^{d}$ is well
defined, i.e, $\mathcal{P}^{d}(\zeta)\in\mathbb{R}$ iff $\zeta\in A_{1}$.
Clearly, \begin{equation}
A_{1}\subset A_{2}:=\left\{ \zeta\in\mathcal{L}^{2}\mid\zeta(x)+\mu\neq0\text{ for a.e. }x\in[a,b]\right\} .\label{r-a2}\end{equation}
 Since \begin{equation}
\frac{(\sigma+\alpha\zeta)^{2}}{\mu+\zeta}+\nu^{-1}\zeta^{2}=\frac{\beta^{2}}{\zeta+\mu}-2\alpha\beta+\alpha^{2}(\zeta+\mu)+\nu^{-1}\zeta^{2},\label{pdsc}\end{equation}
 we have that, \begin{equation}
\mathcal{P}^{d}(\zeta)=2\pi\int_{a}^{b}rh_{-\beta(r)}(r)dr,\label{pd}\end{equation}
 where $h_{-\beta(r)}$ is given by (\ref{def-h}) or (\ref{def-h0}).

Taking into account the expressions of $\widehat{\mathcal{P}}$ in
(\ref{ph}) and $\mathcal{P}^{d}$ in (\ref{pd}) and after applying
Proposition \ref{prop1} for $\tau=-\beta(r)$ we get the following
result.

\begin{proposition} \label{prop2}Assume that $2\mu<\nu\alpha^{2}$.

\emph{(a)} If $\beta^{2}>\eta$ on $[a,b]$, (that is, $\beta^{2}(r)>\eta$
for every $r\in[a,b]$) then equation (\ref{ec-f}) corresponding
to $\tau=-\beta$ has a unique solution $\overline{\zeta}\in C[a,b]\cap A_{1}$
with $\overline{\zeta}>-\mu$ on $[a,b]$, and for $\overline{v}:=r^{-1}\left(\alpha-\beta/(\overline{\zeta}+\mu)\right)$
one has \begin{equation}
\widehat{\mathcal{P}}\mathcal{(}\overline{v})=\min_{v\in\mathcal{L}^{4}}\widehat{\mathcal{P}}(v)=\max_{\zeta\in A_{1},\zeta>-\mu}\mathcal{P}^{d}(\zeta)=\mathcal{P}^{d}(\overline{\zeta}).\label{r-caz-a}\end{equation}

\emph{(b)} If $0<\beta^{2}<\eta$ on $[a,b]$ then equation (\ref{ec-f})
corresponding to $\tau=-\beta$ has three solutions $\overline{\zeta}_{1},\overline{\zeta}_{2},\overline{\zeta}_{3}\in C[a,b]\cap A_{1}$
with $\overline{\zeta}_{1}>-\mu>\overline{\zeta}_{2}>\rho>\overline{\zeta}_{3}>-\tfrac{1}{2}\nu\alpha^{2}$
on $[a,b]$. Moreover, for $\overline{v}_{i}:=r^{-1}\left(\alpha-\beta/(\overline{\zeta}_{i}+\mu)\right)$
$(i\in\{1,2,3\})$, one has \begin{gather}
\widehat{\mathcal{P}}\mathcal{(}\overline{v}_{1})=\min_{v\in\mathcal{L}^{4}}\widehat{\mathcal{P}}(v)=\max_{\zeta\in A_{1},\zeta>-\mu}\mathcal{P}^{d}(\zeta)=\mathcal{P}^{d}(\overline{\zeta}_{1}),\label{r-caz-b1}\\
\widehat{\mathcal{P}}\mathcal{(}\overline{v}_{2})=\min_{\zeta\in A_{1},\overline{\zeta}_{3}<\zeta<-\mu}\mathcal{P}^{d}(\zeta)=\mathcal{P}^{d}(\overline{\zeta}_{2}),\label{r-caz-b2}\\
\widehat{\mathcal{P}}\mathcal{(}\overline{v}_{3})=\max_{\zeta\in A_{1},\zeta<\overline{\zeta}_{2}}\mathcal{P}^{d}(\zeta)=\mathcal{P}^{d}(\overline{\zeta}_{3}).\label{r-caz-b3}\end{gather}

\end{proposition}

\begin{proof}The existence and continuity of the solutions $\overline{\zeta}$
in case (a) and $\overline{\zeta}_{1}$, $\overline{\zeta}_{2},$
$\overline{\zeta}_{3}$ in case (b) follow from the behavior of $f$
presented in Table \ref{tab1} and the fact that $\beta\in C^{\infty}[a,b]$.
More precisely, the critical point of $f$ are $-\mu$ and $\rho$
with $f(-\mu)=0$, $f(\rho)=\eta$ (see Table \ref{tab-extra}) so
the inverses of $f$ restricted to each of the intervals $(-\tfrac{1}{2}\nu\alpha^{2},\rho)$,
$(\rho,-\mu)$, $(-\mu,\infty)$ exist and are smooth. Since for $\beta^{2}>0$
the solution $\varsigma(r)$ of the equation $f(\varsigma)=\beta^{2}(r)$
is different from $-\mu$, we have that $\overline{\zeta},\overline{\zeta}_{1},\overline{\zeta}_{2},\overline{\zeta}_{3}\in A_{1}$
and $\overline{v},\overline{v}_{1},\overline{v}_{2},\overline{v}_{3}$
are continuous and implicitly belong to $\mathcal{L}^{4}.$

(a) For every $v\in\mathcal{L}^{4}$ and $\zeta\in A_{1}$ with $\zeta>-\mu$
on $[a,b]$, apply Proposition \ref{prop1}(a) for a fixed $r\in[a,b]$,
$u:=rv(r)$, and $\varsigma:=\zeta(r)$ to obtain \[
p_{-\beta(r)}(rv(r))\geq p_{-\beta(r)}(r\overline{v}(r))=h_{-\beta(r)}(\overline{\zeta}(r))\geq h_{-\beta(r)}(\zeta(r)).\]
 Multiplying by $2\pi r>0$ and integrating on $[a,b]$ provide \[
\widehat{\mathcal{P}}\mathcal{(}v)\geq\widehat{\mathcal{P}}\mathcal{(}\overline{v})=\mathcal{P}^{d}(\overline{\zeta})\geq\mathcal{P}^{d}(\zeta)\]
 for all $v\in\mathcal{L}^{4}$ and $\zeta\in A_{1}$ with $\zeta>-\mu$.
The conclusion follows. The proof for (b) follows similarly from Proposition
\ref{prop1}(c).\end{proof}

\section{Some comments on \cite[Th.\ 4]{Gao/Ogden:08}}

On \cite[p. 502]{Gao/Ogden:08} one says:

$\bullet$ {}``we consider the kinetically admissible space to be
defined by

$\mathcal{X}_{a}=\{g(r)\in\mathcal{C}[a,b]\mid g^{\prime}\in\mathcal{L}^{p}[a,b],~g(a)=0\},\quad(19)$

\noindent where $\mathcal{L}^{p}$ is the space of Lebesgue integrable
functions for some $p\in[1,\infty)$.''

\noindent followed by:

$\bullet$ {}``Then the considered problem can be formulated as the
following minimization problem for the determination of the deformation
function $g$:

$\min\limits _{g\in\mathcal{X}_{a}}\left\{ \mathcal{P}(g)=2\pi{\displaystyle \int_{a}^{b}}\hat{W}(rg^{\prime}(r))rdr-2\pi b^{2}\tau_{\theta}g(b)\right\} $.
$\quad$(20)''

\medskip{}

\noindent Later on (see \cite[p.\ 506]{Gao/Ogden:08}) one says:

$\bullet$ {}``In this section the strain energy is assumed to be
the nonconvex function of the shear strain $\gamma$ given by

$\hat{W}(\gamma)=\tfrac{1}{2}\mu\gamma^{2}+\tfrac{1}{2}\nu\left(\tfrac{1}{2}\gamma^{2}-\alpha\gamma\right)^{2}$,
$\quad$(39)

\noindent where $\mu>0$, $\nu>0$ and $\alpha\in\mathbb{R}$ are
material constants.''

\medskip{}

\noindent On the next page one continues with

$\bullet$ {}``... there is no loss of generality in restricting
attention to $\alpha>0$, which we do here.'',

$\bullet$ {}``The parameters $\alpha,\mu,\nu$ are then such that

$2\mu<\nu\alpha^{2}<8\mu$, $\quad$(40)

\noindent and for the most part we restrict attention to this range
of values.'',

$\bullet$ {}``In this section, therefore, we restrict attention
to the more interesting case for which $\tau_{\theta}>0$, and hence
$\gamma>0$.

For this situation we define a two-component canonical strain measure
$\mathbf{\xi}$ by

$\mathbf{\xi}=(\varepsilon,\xi)=\Lambda(g)=\left(\tfrac{1}{2}(rg^{\prime})^{2},\tfrac{1}{2}(rg^{\prime})^{2}-\alpha(rg^{\prime})\right)\in\mathbb{R}^{2},$

\noindent and the canonical energy $U(\mathbf{\xi})=\mu\varepsilon+\tfrac{1}{2}\nu\xi^{2}$
is then a convex (quadratic) function, which is well defined on the
domain

$\mathcal{E}=\{(\varepsilon,\xi)\in\mathcal{L}^{1}\times\mathcal{L}^{2}\mid\varepsilon(r)\geq0,~\xi(r)\geq-\tfrac{1}{2}\alpha^{2},\ \forall r\in[a,b]\}$.

\noindent The canonical dual `stress' vector

$\mathbf{\zeta}=(\varsigma,\zeta)=U_{\mathbf{\xi}}(\mathbf{\xi})=(\mu,\nu\xi)\in\mathbb{R}^{2},$

\noindent where $U_{\mathbf{\xi}}(\mathbf{\xi})=\partial U/\partial\xi$,
is well defined on the dual space

$\mathcal{S}=\{(\varsigma,\zeta)\in\mathcal{L}^{\infty}\times\mathcal{L}^{2}\mid\varsigma(r)=\mu,~\zeta(r)\geq-\tfrac{1}{2}\nu\alpha^{2},forallr\in[a,b]\}$.''

\noindent and

$\bullet$ {}``... we obtain the total complementary energy $X(g,\zeta)$
for this nonconvex problem in the form

$X(g,\zeta)=...=2\pi\int_{a}^{b}\left[\tfrac{1}{2}(rg^{\prime})^{2}(\zeta+\mu)-\alpha
rg^{\prime}\zeta-\tfrac{1}{2}\nu^{-1}\zeta^{2}\right]rdr-2\pi
b^{2}\tau_{\theta}g(b)$. $\ $(41)''

\medskip{}

\noindent Page 509 of \cite{Gao/Ogden:08} begins with:

$\bullet$ {}``For a given $\mathbf{\zeta}\in\mathcal{S}$, the criticality
condition $\delta_{g}X(g,\zeta)=0$ leads to the equation

$\left((\zeta+\mu)r^{3}g^{\prime}-r^{2}\alpha\zeta\right)^{\prime}=0,\quad r\in(a,b),$

\noindent and the boundary condition

$(\zeta+\mu)r^{3}g^{\prime}-r^{2}\alpha\zeta=b^{2}\tau_{\theta}\quad$on$\quad r=b$.''

\noindent followed by

$\bullet$ {}``Therefore, by substituting $\gamma=rg^{\prime}=(\sigma+\alpha\zeta)/(\zeta+\mu)$
into $X$, the pure complementary energy $\mathcal{P}^{d}$ can be
obtained by the canonical dual transformation

\noindent  $\mathcal{P}^{d}(\zeta)=\left\{
X(g,\zeta)\mid\delta_{g}X(g,\zeta)=0\right\} =-\pi{\displaystyle
\int_{a}^{b}\left(\frac{(\sigma+\alpha\zeta)^{2}}{\zeta+\mu}+\nu^{-1}\zeta^{2}\right)}rdr$,
$\quad$(42)

\noindent which is well defined on the dual feasible space

$\mathcal{S}_{a}=\left\{ \zeta\in\mathcal{L}^{2}\mid\zeta(r)+\mu\neq0,\right\} .$

\noindent The criticality condition $\delta\mathcal{P}^{d}(\zeta)=0$
leads to the dual algebraic equation

$\left(2\nu^{-1}\zeta+\alpha^{2}\right)(\mu+\zeta)^{2}=(\sigma-\mu\alpha)^{2}$.
$\quad$(43)''.

\medskip{}

\noindent On page 510 of \cite{Gao/Ogden:08} one can find:

$\bullet$ {}``... for simplicity of expression, we have introduced
the notations $\beta=\mu\alpha-\sigma$ and

$\eta=(\nu\alpha^{2}-2\mu)^{3}/27\nu$. $\quad$(48)''

\medskip{}

Note that $\mathcal{P}^{d}(\zeta)$ is exactly as in (\ref{pd-first}).

\medskip{}

Before quoting the result we have in view let us shortly discuss the
quoted text above.

Note that $\sigma(r)=b^{2}\tau_{\theta}/r^{2}$ (as taken at the beginning
of our Section \ref{sec2}), is mentioned on page 509 of \cite{Gao/Ogden:08}
(see also \cite[page 505]{Gao/Ogden:08}); probably one takes $0<a<b<\infty,$
assumptions that we accept here. Hence the set $B_{0}:=\{s\in[a,b]\mid\beta(s)=0\}$
has at most one element.

Probably, by {}``well defined on ... $\mathcal{S}_{a}$'' above
one means that $\mathcal{P}^{d}(\zeta)\in\mathbb{R}$ for every $\zeta\in\mathcal{S}_{a},$
that is, $\mathcal{S}_{a}\subset A_{1}$ {[}see (\ref{r-a1}){]}.
But this inclusion is false as we will see from the next result. However
$\mathcal{S}_{a}\subset A_{2}$.

\begin{lemma} \label{dom-gresit}Under the current notations and
assumptions $\mathcal{S}_{a}\not\subset A_{1}$ and $\mathcal{P}^{d}$
is not well defined on $\mathcal{S}_{a}$.\end{lemma}

\begin{proof} Recall that $\beta(r)=\alpha\mu-b^{2}\tau_{\theta}/r^{2}$
and note that $B_{0}=\{s\in[a,b]\mid\beta(s)=0\}\subset\{\sqrt{b^{2}\tau_{\theta}/(\alpha\mu)}\}$
so there exist $(c,d)\subset(a,b)$ and $\gamma>0$ such that $\beta^{2}(x)\geq\gamma$
for every $x\in[c,d]$. Let $\zeta(r)=-\mu+r-c$ for $r\in(c,d)$,
$\zeta(r)=1-\mu$ for $r\in[a,b]\setminus(c,d)$. Then $\zeta\in\mathcal{L}^{2}$
and $\zeta>-\mu>-\tfrac{1}{2}\nu\alpha^{2}$ on $[a,b]$, i.e., $\zeta\in\mathcal{S}_{a}$.
Also $\frac{\beta^{2}}{\zeta+\mu}\ge\frac{\gamma}{\zeta+\mu}>0$ on
$[c,d]$ and $\int_{a}^{b}\frac{\beta^{2}(r)}{\zeta(r)+\mu}dr\ge\gamma\int_{c}^{d}\frac{dr}{r-c}=+\infty$,
that is $\zeta\not\in A_{1}$. \end{proof}

\medskip{}

Let us denote the algebraic interior (or core) of a set by {}``$\operatorname*{core}$''.

\begin{lemma} \label{empty-core}Under the current notations and
assumptions $\operatorname*{core}A_{2}$ is empty. In particular,
$\operatorname*{core}A_{1}=\operatorname*{core}\mathcal{S}_{a}=\emptyset$.\end{lemma}

\begin{proof} Indeed, because $\sigma$ is continuous, there exist
$\overline{\delta}>0$ and $a\leq a^{\prime}<b^{\prime}\leq b$ such
that $\left\vert \beta\left(r\right)\right\vert \geq\overline{\delta}>0$
for every $r\in I=(a^{\prime},b^{\prime})$. Take $\overline{\zeta}\in A_{2}$
and define $u$ by $u(r)=n\big[a^{\prime}+\frac{b^{\prime}-a^{\prime}}{n}-r-\overline{\zeta}(r)-\mu\big]$
for $r\in[a^{\prime}+\frac{b^{\prime}-a^{\prime}}{n+1},a^{\prime}+\frac{b^{\prime}-a^{\prime}}{n})$
with $n\geq1$, and $u(r)=0$ for $r\in[a,b]\setminus I$. Then for
every $\delta>0$ there exists $t=\frac{1}{n}\in(0,\delta)$ such
that $\overline{\zeta}+tu\notin A_{2}$. (Note that for $\overline{\zeta}\in A_{1}$
and for $u$ constructed as above we have that for every $\delta>0$
there exists $t\in(0,\delta)$ such that $\overline{\zeta}+tu\notin A_{1}$).
\end{proof}

\bigskip{}

To our knowledge, one can speak about G\^{a}teaux differentiability of
a function $f:E\subset X\rightarrow Y$, with $X,Y$ topological vector
spaces, at $\overline{x}\in E$ only if $\overline{x}$ is in the
core of $E$. As we have seen above, $\mathcal{P}^{d}(\zeta)\in\mathbb{R}$
only for $\zeta\in A_{1}$ and $\operatorname*{core}A_{1}=\emptyset.$
These considerations naturally lead to the following question:

\medskip{}

\emph{In what sense is the critical point notion associated to $\mathcal{P}^{d}$
understood so that when using this notion one gets \cite[(43)]{Gao/Ogden:08},
other than just formal computation?}

\medskip{}

Taking into account the comment (see \cite[page 509]{Gao/Ogden:08})

$\bullet$ {}``We emphasize that the integrand of $\mathcal{P}^{d}(\zeta)$
has a singularity at $\zeta=-\mu$, which is excluded in the definition
of $\mathcal{S}_{a}$, and does not in general correspond to a critical
point of $\mathcal{P}^{d}(\zeta)$'',

\noindent we wish to point out that there is an important difference
between the condition $\zeta\neq-\mu$ and $\zeta(r)\neq-\mu$ a.e.
on $[a,b]$ since it is known that $\zeta\neq-\mu$ means that $\zeta(x)\neq-\mu$
on a set of positive measure. Alternatively, $\mathcal{L}^{2}\setminus\{-\mu\}$
is a (nonempty) open set, while, as seen above, the set $A_{2}:=\left\{ \zeta\in\mathcal{L}^{2}\mid\zeta(r)+\mu\neq0\text{ for a.e. }r\in[a,b]\right\} $
has empty core (in particular has empty interior).

\medskip{}

From the quoted text it seems that $\zeta$ is taken from $\mathcal{L}^{2}$,
but it is not very clear from where $g$ is taken; apparently $g\in\mathcal{X}_{a}$,
that is, $g\in\mathcal{C}[a,b]$ is such that $g^{\prime}\in\mathcal{L}^{p}$
and $g(a)=0$ for some fixed $p\in[1,\infty)$. But, to have $(\varepsilon,\xi)\in\mathcal{L}^{1}\times\mathcal{L}^{2}$
where $\varepsilon(r)=\tfrac{1}{2}(rg^{\prime}(r))^{2}$ and $\xi(r)=\tfrac{1}{2}(rg^{\prime}(r))^{2}-\alpha rg^{\prime}(r)$
one needs $p\geq4$. Moreover, the statement $\left(\tfrac{1}{2}(rg^{\prime})^{2},\tfrac{1}{2}(rg^{\prime})^{2}-\alpha(rg^{\prime})\right)\in\mathbb{R}^{2}$
is quite strange because $g^{\prime}$ is a function (or even an equivalent
class).

\medskip{}

In the sequel we take $p=4$. Then $g$ in $\mathcal{X}_{a}$ has
to be an absolutely continuous function on $[a,b]$ with $g(a)=0$
and $g^{\prime}\in\mathcal{L}^{4}[a,b]$. In fact \[
g\in\mathcal{X}_{a}\Longleftrightarrow\exists v\in\mathcal{L}^{4},\ {\rm for\ all}\ r\in[a,b]:g(r)=\int_{a}^{r}v(t)dt.\]
 So, the problem $\min_{g\in\mathcal{X}_{a}}\mathcal{P}(g)$ above
becomes the problem $\min_{v\in\mathcal{L}^{4}}\widehat{\mathcal{P}}(v)$
with $\widehat{\mathcal{P}}(v)$ defined in (\ref{ph}).

Next, using the above considerations we discuss the following result
of \cite{Gao/Ogden:08}; we also quote its proof for easy reference.

\medskip{}

{}``\textbf{Theorem 4 (Extremality Criteria)}. For a given shear
stress $\tau_{\theta}>0$ such that $\sigma=b^{2}\tau_{\theta}/r^{2}$,
if $(\sigma-\alpha\mu)^{2}>\eta>0$ the dual algebraic equation (43)
has a unique real root $\overline{\zeta}(r)>-\mu$, which is a global
maximizer of $\mathcal{P}^{d}$ over $\mathcal{S}_{a}$, and the corresponding
solution $\overline{g}$ is a global minimizer of $\mathcal{P}(g)$
over $\mathcal{X}_{a}$, i.e.

$\mathcal{P(}\overline{g})=\min\limits _{g\in\mathcal{X}_{a}}\mathcal{P}(g)=\max\limits _{\zeta\in\mathcal{S}_{a}}\mathcal{P}^{d}(\zeta)=\mathcal{P}^{d}(\overline{\zeta})$.
$\quad$(51)

If $(\sigma-\alpha\mu)^{2}<\eta$ and $\sigma\neq\alpha\mu$ then
equation (43) has three real roots ordered as (49). The corresponding
solution $\overline{g}_{1}$ is a global minimizer of $\mathcal{P}(g)$
and $\overline{\zeta}_{1}$ is a global maximizer of $\mathcal{P}^{d}(\zeta)$
over the domain $\zeta>-\mu$, i.e.

$\mathcal{P(}\overline{g})=\min\limits _{g\in\mathcal{X}_{a}}\mathcal{P}(g)=\max\limits _{\zeta>-\mu}\mathcal{P}^{d}(\zeta)=\mathcal{P}^{d}(\overline{\zeta}_{1})$.
$\quad$(52)

For $\overline{\zeta}_{2}$, the corresponding solution $\overline{g}_{2}$
is a local minimizer of $\mathcal{P}(g)$, while for $\overline{\zeta}_{3}$
the associated $\overline{g}_{3}$ is a local maximizer, so that

$\mathcal{P(}\overline{g}_{2})=\min\limits _{g\in\mathcal{X}_{2}}\mathcal{P}(g)=\min\limits _{\overline{\zeta}_{3}<\zeta<-\mu}\mathcal{P}^{d}(\zeta)=\mathcal{P}^{d}(\overline{\zeta}_{2})$,
$\quad$(53)

\noindent and

$\mathcal{P(}\overline{g}_{3})=\max\limits _{g\in\mathcal{X}_{3}}\mathcal{P}(g)=\max\limits _{-\tfrac{1}{2}\nu\alpha^{2}<\zeta<\overline{\zeta}_{2}}\mathcal{P}^{d}(\zeta)=\mathcal{P}^{d}(\overline{\zeta}_{3})$,
$\quad$(54)

\noindent where $\mathcal{X}_{i}$ is a neighborhood of $\overline{g}_{i}$
for $i=2,3$.

In the transitional cases for which $\sigma-\alpha\mu=\pm\sqrt{\eta}$,
the two roots $\overline{\zeta}_{2}$ and $\overline{\zeta}_{3}$
coincide and the local extrema described by (53) and (54) merge into
a horizontal point of inflection. In the case $\sigma=\alpha\mu$
the two roots $\overline{\zeta}_{1}$ and $\overline{\zeta}_{2}$
coincide at $-\mu$ and this is the transitional point at which the
two local minima are equal and the solution becomes nonsmooth.

\bigskip{}

Proof. The proof of this theorem follows from the triality theory
developed in {[}5, 8{]}. $\quad\square$''

\medskip{}

The references {[}5, 8{]} above are our references \cite{Gao:98}
and \cite{Gaoo-book}, respectively.

\medskip{}

Before discussing the above result let us clarify the meaning for
$\overline{\zeta}_{i}$ and $\overline{g}_{i}$ (as well as $\overline{\zeta}$
and $\overline{g}$) appearing in the above statement. In fact these
functions are defined in the statement of \cite[Th.~2]{Gao/Ogden:08}:

\medskip{}

{}``\textbf{Theorem 2 (Closed-form Solutions).} For a given shear
stress $\tau_{\theta}>0$ such that $\sigma=b^{2}\tau_{\theta}/r^{2}$,
the dual algebraic equation (43) has at most three real roots $\overline{\zeta}_{i}(r)$,
$i=1,2,3$, ordered as

$\overline{\zeta}_{1}(r)\geq-\mu\geq\overline{\zeta}_{2}(r)\geq\overline{\zeta}_{3}(r)\geq-\tfrac{1}{2}\nu\alpha^{2}$.
$\quad$(49)

\noindent For each of these roots, the function defined by

$\overline{g}_{i}(r)=\int_{a}^{r}\frac{b^{2}\tau_{\theta}/s^{2}+\alpha\overline{\zeta}_{i}(s)}{s\left(\overline{\zeta}_{i}(s)+\mu\right)}ds\quad$
(50)

\noindent is a solution of the boundary value problem (BVP), and

$\mathcal{P}(\overline{g}_{i})=\mathcal{P}^{d}(\overline{\zeta}_{i}),~~i=1,2,3$.''

\medskip{}

With our reformulation of the problem \cite[(20)]{Gao/Ogden:08} in
place, in the statements of \cite[Th.~2, Th.\ 3]{Gao/Ogden:08} one
must replace $\mathcal{X}_{a}$ by $\mathcal{L}^{4}$, $\overline{g}_{i}$
by $\overline{v}_{i}:=\frac{\sigma+\alpha\overline{\zeta}_{i}}{r(\overline{\zeta}_{i}+\mu)},$
$\overline{g}$ by $\overline{v}:=\frac{\sigma+\alpha\overline{\zeta}}{r(\overline{\zeta}+\mu)}$
and $\mathcal{P}$ by $\widehat{\mathcal{P}}$, $\mathcal{X}_{j}$
being now a neighborhood of $\overline{v}_{j}$, for $j=2,3$. This
is possible since the operator $v\in\mathcal{L}^{4}\rightarrow u=\int_{a}^{r}v\in\mathcal{X}_{a}$
and its inverse $\mathcal{X}_{a}\ni u\rightarrow v=u_{x}\in\mathcal{L}^{4}$
are linear continuous under the $W^{1,4}$ topology on $\mathcal{X}_{a}$;
whence $g\in\mathcal{X}_{a}$ is a local extrema for $\mathcal{P}$
iff the corresponding $v\in\mathcal{L}^{4}$ is a local extrema for
$\widehat{\mathcal{P}}$.

Hence $\overline{\zeta}$, $\overline{\zeta}_{i}$ and $\overline{v},$
$\overline{v}_{i}$ are exactly as in our Proposition \ref{prop2}.
We observe that there are differences between the conclusions of Proposition
\ref{prop2} and \cite[Th.\ 4]{Gao/Ogden:08}, the conclusions in
\cite[Th.\ 4]{Gao/Ogden:08} being stronger. In the discussion below
we show that it is not possible to obtain stronger conclusions than
those of Proposition \ref{prop2}.

\medskip{}

\emph{Discussion of \cite[(51)]{Gao/Ogden:08}.} The sole difference
between \cite[(51)]{Gao/Ogden:08} and (\ref{r-caz-a}) is that one
has $\max_{\zeta\in\mathcal{S}_{a}}\mathcal{P}^{d}(\zeta)$ instead
of $\max_{\zeta\in A_{1},\zeta>-\mu}\mathcal{P}^{d}(\zeta)$. As seen
Section \ref{sec2}, $\mathcal{P}^{d}(\zeta)\in\mathbb{R}$ only for
$\zeta\in A_{1}$, so considering $\sup_{\zeta\in\mathcal{S}_{a}}\mathcal{P}^{d}(\zeta)$
makes no sense. One can ask if \cite[(51)]{Gao/Ogden:08} holds when
one replaces $\max_{\zeta\in\mathcal{S}_{a}}\mathcal{P}^{d}(\zeta)$
by $\max_{\zeta\in A_{1}^{0}}\mathcal{P}^{d}(\zeta)$, where \[
A_{1}^{0}:=\{\zeta\in A_{1}\mid\zeta\geq-\tfrac{1}{2}\nu\alpha^{2}\}.\]
 That is not true. Indeed, consider $\zeta_{n}(r)=-\mu-\gamma(r-a)$
for $r\in[a+(b-a)/n,b]$ and $\zeta_{n}(r)=-\mu-\gamma(b-a)/n$ for
$r\in[a,a+(b-a)/n)$, where $0<\gamma<\left(\tfrac{1}{2}\nu\alpha^{2}-\mu\right)/(b-a)$.
Clearly $-\mu-\gamma(b-a)/n\geq\zeta_{n}\geq-\mu-\gamma(b-a)>-\tfrac{1}{2}\nu\alpha^{2}$
on $[a,b]$, and so $\zeta_{n}\in A_{1}^{0}$. Moreover \[
-\int_{a}^{b}\frac{\beta^{2}}{\zeta_{n}+\mu}dr\geq\int_{a+(b-a)/n}^{b}\frac{\beta^{2}(r)}{\gamma(r-a)}dr\geq\frac{\eta}{\gamma}\ln n\rightarrow\infty,\]
 which proves that $\sup_{\zeta\in A_{1}^{0}}\mathcal{P}^{d}(\zeta)=+\infty$
(see (\ref{pd-first})).

\medskip{}

\emph{Discussion of \cite[(52)]{Gao/Ogden:08}.} A similar discussion
as above shows that $\max_{\zeta>-\mu}\mathcal{P}^{d}(\zeta)$ in
\cite[(52)]{Gao/Ogden:08} does not make sense. The correct equality
has been established in (\ref{r-caz-b1}).

\medskip{}

\emph{Discussion of \cite[(53)]{Gao/Ogden:08}.} Assume that $0<\beta^{2}<\eta$
on $[a,b]$. It is easy to show that $\{\zeta\in\mathcal{L}^{2}\mid\rho<\zeta<-\mu\}\not\subset A_{1}$,
which proves that $\{\zeta\in\mathcal{L}^{2}\mid\overline{\zeta}_{3}<\zeta<-\mu\}\not\subset A_{1}$;
take for example $\zeta(r)=-\mu+\frac{\rho+\mu}{\eta}(r-a)\beta^{2}(r)$,
$r\in(a,b)$. This shows that $\min_{\overline{\zeta}_{3}<\zeta<-\mu}\mathcal{P}^{d}(\zeta)$
in \cite[(53)]{Gao/Ogden:08} does not make sense. Therefore one has
to replace the set $\{\zeta\in\mathcal{L}^{2}\mid\overline{\zeta}_{3}<\zeta<-\mu\}$
by $\{\zeta\in A_{1}\mid\overline{\zeta}_{3}<\zeta<-\mu\}$, that
is, as in (\ref{r-caz-b2}).

The problem that occurs now is whether $\overline{v}_{2}$ is a local
minimum point of $\widehat{\mathcal{P}}$. The answer for this problem
is negative. Indeed, consider $\varepsilon\in(0,b-a)$ and take $v_{\varepsilon}(r):=\overline{v}_{1}(r)$
for $r\in[a,a+\varepsilon]$, $v_{\varepsilon}(r):=\overline{v}_{2}(r)$
for $r\in(a+\varepsilon,b]$. From Corollary \ref{cor1} we have that
$p_{-\beta(r)}(rv_{\varepsilon}(r))<p_{-\beta(r)}(r\overline{v}_{2}(r))$
for every $r\in[a,a+\varepsilon]$, and so \[
\widehat{\mathcal{P}}(\overline{v}_{2})-\widehat{\mathcal{P}}(v_{\varepsilon})=2\pi\int_{a}^{a+\varepsilon}r\left[p_{-\beta(r)}(r\overline{v}_{2}(r))-p_{-\beta(r)}(rv_{\varepsilon}(r))\right]dr>0.\]
 Since $\left\Vert v_{\varepsilon}-\overline{v}_{2}\right\Vert =\big(\int_{a}^{a+\varepsilon}\left\vert \overline{v}_{2}-\overline{v}_{1}\right\vert ^{4}\big)^{1/4}\leq\varepsilon^{1/4}\left\Vert \overline{v}_{2}-\overline{v}_{1}\right\Vert ,$
it is clear that $\overline{v}_{2}$ is not a local minimum point
of $\widehat{\mathcal{P}}.$

\medskip{}

\emph{Discussion of \cite[(54)]{Gao/Ogden:08}.} Clearly, $\{\zeta\in\mathcal{L}^{2}\mid-\tfrac{1}{2}\nu\alpha^{2}<\zeta<\overline{\zeta}_{2}\}\subset A_{1};$
hence the last equality in \cite[(54)]{Gao/Ogden:08} holds due to
Proposition \ref{prop2} (b) and relation (\ref{r-caz-b3}). As above,
now the problem is whether $\overline{v}_{3}$ is a local maximum
point of $\widehat{\mathcal{P}}$. The answer is negative. Indeed,
as seen in Proposition \ref{prop2}, $\overline{\zeta}_{3}\in C[a,b]$,
$\overline{\zeta}_{3}<\rho<-\mu$ on $[a,b]$ so $\overline{v}_{3}\in C[a,b]$.
It follows that the mapping $q:[a,b]\rightarrow\mathbb{R}$ defined
by $q(r):=p_{-\beta(r)}(r\overline{v}_{3}(r))$ is continuous, and
so $M:=\max_{[a,b]}q\in\mathbb{R}$. Set $m:=\max_{[a,b]}\left\vert \beta\right\vert >0$.
It follows that \begin{align*}
p_{-\beta(r)}(y) & =\tfrac{1}{2}\mu y^{2}+\tfrac{1}{2}\nu\left(\tfrac{1}{2}y^{2}-\alpha y\right)^{2}-\alpha\mu y+\beta(r)y\\
 & \geq\tfrac{1}{2}\mu y^{2}+\tfrac{1}{2}\nu\left(\tfrac{1}{2}y^{2}-\alpha y\right)^{2}-\alpha\mu y-my=p_{m}(y)\end{align*}
 for all $r\in[a,b]$ and $y\geq0$. Since $\lim_{y\rightarrow\infty}p_{m}(y)=\infty$,
there exists $y_{0}>0$ such that $p_{m}(y)>M$ for all $y\geq y_{0}$.
Let us take $\overline{y}>0$ such that $a\overline{y}\geq y_{0}$
(and so $r\overline{y}\geq y_{0}$ for every $r\in[a,b]$). Consider
$\varepsilon\in(0,b-a)$ and $v_{\varepsilon}(r):=\overline{y}$ for
$r\in[a,a+\varepsilon]$, $v_{\varepsilon}(r):=\overline{v}_{3}(r)$
for $r\in(a+\varepsilon,b]$. We have that $p_{-\beta(r)}(r\overline{v}_{3}(r))\leq M<p_{m}(r\overline{y})\leq p_{-\beta(r)}(r\overline{y})$
for every $r\in[a,a+\varepsilon]$, and so \[
\widehat{\mathcal{P}}(\overline{v}_{3})-\widehat{\mathcal{P}}(v_{\varepsilon})=2\pi\int_{a}^{a+\varepsilon}r\left[p_{-\beta(r)}(r\overline{v}_{3}(r))-p_{-\beta(r)}(r\overline{y})\right]dr<0.\]
 Since $\left\Vert v_{\varepsilon}-\overline{v}_{3}\right\Vert =\big(\int_{a}^{a+\varepsilon}\left\vert \overline{v}_{3}-\overline{y}\right\vert ^{4}\big)^{1/4}\leq\varepsilon^{1/4}\left\Vert \overline{v}_{3}-\overline{y}\right\Vert $,
this proves that $\overline{v}_{3}$ is not a local maximum of $\widehat{\mathcal{P}}.$

\section{Conclusions}
\begin{itemize}
\item The function $\mathcal{P}^{d}$ is not well defined on the set $\mathcal{S}_{a}.$
The biggest set on which $\mathcal{P}^{d}$ is well defined is $A_{1}$
whose core is empty; hence it is not possible to speak about critical
points using the G\^{a}teaux differential.
\item Since $\mathcal{P}^{d}$ is not well defined on the sets $\mathcal{S}_{a},$
$\{\zeta\in\mathcal{L}^{2}\mid\zeta>-\mu\}$ and $\{\zeta\in\mathcal{L}^{2}\mid\overline{\zeta}_{3}<\zeta<-\mu\}$
the quantities $\max_{\zeta\in\mathcal{S}_{a}}\mathcal{P}^{d}(\zeta),$
$\max_{\zeta>-\mu}\mathcal{P}^{d}(\zeta)$ and $\min_{\overline{\zeta}_{3}<\zeta<-\mu}\mathcal{P}^{d}(\zeta)$
in \cite[Th.\ 4]{Gao/Ogden:08} make no sense.
\item The element $\overline{g}_{2}$ is not a local minimizer of $\mathcal{P}$
and $\overline{g}_{3}$ is not a local maximizer of $\mathcal{P}.$
\item As seen above, one says that the proof of \cite[Th.\ 4]{Gao/Ogden:08}
{}``follows from the triality theory'', without mentioning a precise
result. In fact, as seen in Section \ref{sec2}, the proof of the
correct variant of \cite[Th.\ 4]{Gao/Ogden:08}, that is, Proposition
\ref{prop2}, follows from an elementary result, while \cite[Th.\ 4]{Gao/Ogden:08}
is obtained by a simple extrapolation of Proposition \ref{prop1}. \end{itemize}


\begin{thebibliography}{4}
\bibitem{Gao:98}D. Y. Gao, Duality, triality and complementary extremum
principles in non-convex parametric variational problems with applications,
IMA J. Appl. Math. 61 (1998) 199--235.

\bibitem{Gaoo-book}D. Y. Gao, Duality Principles in Nonconvex Systems:
Theory, Methods and Applications (Kluwer, Dordrecht 2000).

\bibitem{Gao/Ogden:08}D. Y. Gao, R. W. Ogden, Closed-form solutions,
extremality and nonsmoothness criteria in a large deformation elasticity
problem, Z. angew. Math. Phys. 59 (2008), 498--517.

\bibitem{Gao/Strang:89}D. Y. Gao, G. Strang, \emph{Geometric nonlinearity:
Potential energy, complementary energy, and the gap function}, Quart.
Appl. Math. 47 (1989), 487--504.
\end{thebibliography}
\end{document}